\documentclass[12pt]{amsart}
\usepackage[babel]{csquotes}
\usepackage{enumitem}
\usepackage{amsmath,amsthm,amssymb,mathrsfs,amsfonts,verbatim,enumitem,color,leftidx}
\usepackage{kotex}
\usepackage{mathabx}
\usepackage{graphicx}
\usepackage[left=2.9cm,right=2.9cm,top=3.3cm,bottom=3.4cm,a4paper]{geometry}
\usepackage{etoolbox} 
\usepackage{tikz}
\usepackage{bbm}
\usepackage[all,tips]{xy}
\usepackage{graphicx,ifpdf}
\usepackage{stmaryrd}
\ifpdf
\fi

\usepackage[right,displaymath,mathlines]{lineno}
\allowdisplaybreaks[0]

\definecolor{dartmouthgreen}{rgb}{0.05, 0.5, 0.06}

\usepackage[colorlinks]{hyperref}
\hypersetup{
linkcolor=blue,          
citecolor=dartmouthgreen,      
}

\usepackage{tikz}
\usepackage{here}
\usetikzlibrary{arrows,snakes,backgrounds}

\newtheorem{thm}{Theorem}[section]
\newtheorem{lem}[thm]{Lemma}
\newtheorem{coro}[thm]{Corollary}
\newtheorem{prop}[thm]{Proposition}

\theoremstyle{definition}
\newtheorem{defn}[thm]{Definition}

\newtheorem{remark}[thm]{Remark}

\theoremstyle{remark}

\numberwithin{equation}{section}

\definecolor{esperance}{rgb}{0.0,0.5,0.0}

\title[Spectrum of non-uniform weighted complex]{Spectrum of weighted adjacency operator on a non-uniform arithmetic quotient of $PGL_3$}

\begin{document}

\begin{abstract}
We investigate the automorphic spectra of the natural weighted adjacency operator on the complex arising as a $PGL(3,\mathbb{F}_q[t])$ quotient of $\widetilde{A}_2$-type building. We prove that the set of non-trivial approximate eigenvalues $(\lambda^+,\lambda^-)$ of the weighted adjacency operators $A_w^\pm$ on the quotient induced from the colored adjacency operators $A^\pm$ on the building for $PGL_3$ contains the simultaneous spectrum of $A^\pm$ and another hypocycloid with three cusps. As a byproduct, we re-establish a proof of the fact that $PGL(3,\mathbb{F}_q[t])\backslash PGL(3,\mathbb{F}_q(\!(t^{-1})\!))/PGL(3,\mathbb{F}_q[\![t^{-1}]\!])$ is not a Ramanujan complex, from a combinatorial aspect.

\end{abstract}


\author{Soonki Hong}
\address{Soonki Hong}
\curraddr{Department of Mathematical Education\\ Catholic Kwandong University \\ Gangneung 25601 \\ Republic of Korea}
\email{soonki.hong@snu.ac.kr}

\author{Sanghoon Kwon}
\address{Sanghoon Kwon*}
\curraddr{Room 506 Department of Mathematical Education\\ Catholic Kwandong University \\ Gangneung 25601 \\ Republic of Korea}
\email{shkwon1988@gmail.com \\ skwon@cku.ac.kr}

\thanks{2020 \emph{Mathematics Subject Classification.} Primary 20E42, 20G25; Secondary 47A25}


\maketitle
\tableofcontents

\section{Introduction}\label{sec:1}

A finite $k$-regular graph $X$ is called a \emph{Ramanujan graph} if for every eigenvalue $\lambda$ of the adjacency matrix $A_X$ of $X$ satisfies either $\lambda=\pm k$ or $|\lambda|\le 2\sqrt{k-1}$. An eigenvalue $\lambda$ is called trivial if $\lambda=\pm k$. Since the interval $[-2\sqrt{k-1},2\sqrt{k-1}]$ is equal to the spectrum $\mathcal{S}^2$ of the adjacency operator of $k$-regular tree $\mathcal{T}_k$, we note that a finite $k$-regular graph $X$ is Ramanujan if and only if every {non-trivial} spectrum of $A_X$ is contained in the spectrum $\mathcal{S}^2$ of the adjacency operator $A$ on $L^2(\mathcal{T}_{k})$.

Such graphs can be constructed as quotients of the Bruhat-Tits tree associated to $PGL(2,\mathbb{Q}_p)$ by congruence subgroups of uniform lattices of $PGL(2,\mathbb{Q}_p)$ \cite{LPS}, using Ramanujan conjecture for classical modular forms. More examples were given by Morgenstern \cite{Mo1}, replacing $\mathbb{Q}_p$ by $\mathbb{F}_q(\!(t^{-1})\!)$. One significant difference between $PGL(2,\mathbb{Q}_p)$ and $PGL_2(\mathbb{F}_q(\!(t^{-1})\!))$ is that $PGL(2,\mathbb{F}_q(\!(t^{-1})\!))$ has a non-uniform lattice $\Gamma=PGL(2,\mathbb{F}_q[t])$. For congruence subgroups $\Lambda$ of $\Gamma$, the quotient graphs are infinite but the edges and vertices come with suitable weights $w$ so that the total volume associated to the weight is finite. Under these weights on vertices and edges, the adjacency operator $A$ on $\mathcal{T}_{q+1}$ induces the weighted operator $A_{X}$ on the quotient $X=\Lambda\backslash\mathcal{T}_{q+1}$. 

In \cite{M}, the author defined \emph{Ramanujan diagrams} as such weighted objects satisfying the similar bound for non-trivial spectrum of $A_X$. 
In this case, every non-trivial spectrum of $A_X$ on $L^2_w(X)$ is contained in the interval $[-2\sqrt{q},2\sqrt{q}]$ and hence it is a Ramanujan diagram. For example, the adjacency operator on $L^2_w(PGL(2,\mathbb{F}_q[t])\backslash\mathcal{T}_{q+1})$ has discrete spectrum $\pm(q+1)$ and continuous spectrum $[-2\sqrt{q},2\sqrt{q}]$ (see Figure~\ref{Spec}).

\vspace{2em}
\begin{center}
\begin{figure}[h]
\begin{tikzpicture}
        \draw[dashed,->] (-5,0) -- (5,0);

        \draw[line width=1.2pt,black] (2.5,0) -- (-2.5,0);
            
        \draw[fill] (4,0) circle (0.05);
        \draw[fill] (-4,0) circle (0.05);
        
        \node at (4,-0.4) {$q+1$};
        \node at (-4.2,-0.4) {$-q-1$};
        
        \node at (2.5,-0.4) {$2\sqrt{q}$};
        \node at (-2.6,-0.4) {$-2\sqrt{q}$};     
\end{tikzpicture}
\caption{Spectrum of $A_X$ on $PGL(2,\mathbb{F}_q[t])\backslash \mathcal{T}_{q+1}$}\label{Spec}
\end{figure}
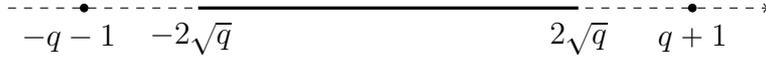
\end{center}

For graphs or diagram coming from a quotient of $G=PGL(2,\mathbb{F}_q(\!(t^{-1})\!))$ by an arithmetic lattice $\Gamma$, being a Ramanujan graph or Ramanujan diagram can be understood via representation-theoretic reformulation. Namely, $\Gamma\backslash \mathcal{T}_{q+1}$ is Ramanujan if and only if all the infinite-dimensional spherical irreducible $G$-representations which are weakly contained in $L^2(\Gamma\backslash G)$ are not from the complementary series. See \cite{Lu} for the detail.

The authors in \cite{CSZ} suggested a generalization of the notion of Ramanujan graphs to the simplicial complexes obtained as finite quotients of the Bruhat-Tits building for $PGL(d,F)$ for a non-Archimedean local field $F$. Let $\mathcal{B}$ be the building associated to $PGL(d,F)$. The \textit{colored adjacency operator} $A_j:L^2(\mathcal{B})\rightarrow L^2(\mathcal{B})$ is defined for $f\in L^2(\mathcal{B})$ by
$$ A_j f(x)=\sum_{\substack{ y\sim x\\ \tau(y)=\tau(x)+j}}f(y),$$
 where $y\sim x$ implies that there is an edge between $y$ and $x$ in $\mathcal{B}$ and $\tau\colon \mathcal{B}^0\to\mathbb{Z}/d\mathbb{Z}$ is a color function (see Section~\ref{sec:2} for the precise definition). Let $\mathcal{S}^d$ be the simultaneous spectrum of colored adjacency operators $(A_1,\ldots,A_{d-1})$ on $L^2(\mathcal{B}^0)$, which may be computed explicitly as a subset of $\mathbb{C}^{d-1}$ (see Theorem~2.11 of \cite{LSV1} and also Proposition~4.5 of \cite{CM} for $d=3$). In fact, $\mathcal{S}^d$ is equal to the set $\sigma(S)$ for 
 $$S=\{(z_1,\ldots,z_d)\colon |z_1|=\cdots=|z_d|=1\textrm{ and }z_1z_2\cdots z_d=1\}$$ and $\sigma\colon S\to\mathbb{C}^{d-1}$ be the map given by $(z_1,\ldots,z_d)\mapsto (\lambda_1,\ldots,\lambda_{d-1})$ where
 $$\lambda_k=q^{\frac{k(d-k)}{2}}\sigma_k(z_1,z_2,\ldots,z_d).$$
A finite complex $X$ arising as a quotient of $\mathcal{B}(G)$ is called \textit{Ramanujan} if every non-trivial automorphic spectrum $(\lambda_1,\ldots,\lambda_{d-1})$ of $A_{X,j}$ acting on $L^2(X)$ is contained in the simultaneous spectrum $\mathcal{S}^d$ of $A_j$ on $\mathcal{B}$.
 In \cite{LSV1}, \cite{LSV2}, \cite{Winnie} and \cite{Sar}, the authors constructed higher dimensional Ramanujan complexes arising as finite quotients of $PGL(d,F)$. 

In \cite{S}, the author investigated non-uniform Ramanujan quotients of the Bruhat-Tits building $\mathcal{B}_d$ of $PGL(d,\mathbb{F}_q(\!(t^{-1})\!))$, generalizing the finite Ramanujan complexes constructed in \cite{LSV1}, \cite{LSV2}, \cite{Winnie} and \cite{Sar}. She proved using the representation-theoretic arugment that if $d>2$, then for $G=PGL(d,\mathbb{F}_q(\!(t^{-1})\!))$, $\Gamma=PGL(d,\mathbb{F}_q[t])$ and $\mathcal{B}_d$ the Bruhat-Tits building of $G$, the quotient $\Lambda\backslash \mathcal{B}_d$ is not Ramanujan for any finite index subgroup $\Lambda$ of $\Gamma$. This mainly comes from the following fact: the Ramanujan conjecture in positive characteristic for $PGL_d$ for $d>2$, achieved by Lafforgue, gives bounds on the cuspidal spectrum, but the other parts of the spectrum do not satisfy the same bounds as the cuspidal spectrum. 

While there are significant differences from the point of view of representation theory, the combinatorial distinction between the cuspidal spectrum and the other parts of the spectrum was not clear, as mentioned in \cite{S}. In this paper, we explore the combinatorial characterization of the automorphic spectrum of the natural weighted adjacency operator on the non-uniform simplicial complex $PGL(3,\mathbb{F}_q[t])\backslash\mathcal{B}_3$.

Let $\Gamma=PGL(3,\mathbb{F}_q[t])$ and $G=PGL(3,\mathbb{F}_q(\!(t^{-1})\!))$. 
Since $\Gamma$ acts with torsion on $\mathcal{B}(G)$ the degree of vertices in $\Gamma\backslash \mathcal{B}(G)$ is not constant, the colored adjacency operators $A^+=A_1$ and $A^-=A_2$ on $\mathcal{B}(G)$ induces the \emph{weighted adjacency operators} $A^+_w$ and $A_w^-$ on $L^2_w(\Gamma\backslash \mathcal{B}(G))$ for a suitable weight function $w$. 

More precisely, the weighted adjacency operators $A_w^+$ and $A_w^-$ on $L_w^2(\Gamma\backslash\mathcal{B}(G))$ is defined for any $f\in L^2_w(\Gamma\backslash\mathcal{B}(G))$ by
$$A_w^\pm f(u):=\sum_{\substack{(u,v)\in E\\ \tau(v)=\tau(u)\pm1}}\frac{w(u,v)}{w(v)}f(v).$$
See Section~\ref{sec:3} for the exact calculation of $w$. These operators $A_w^\pm$ satisfies that $(A_w^+)^*=A_w^-$ (see \cite{S}, Remark 4.1). Now we state our main theorem.

\begin{thm}\label{thm:1.1} Let $A_w^+$ be weighted adjacency operator on $L^2_w(\Gamma\backslash \mathcal{B}(G))$ induced from the colored adjacency operator $A^+$ on $\mathcal{B}(G)$. The spectrum of the operator $A_w^+$ contains
$\Sigma_0\cup\Sigma_1\cup\Sigma_2$ where
$$\Sigma_0=\{q^2+q+1,(q^2+q+1)e^{\frac{2\pi i}{3}},(q^2+q+1)e^{\frac{4\pi i}{3}}\}$$ is a set of three distinct points,
$$\Sigma_1=\{q^{\frac{3}{2}}e^{i\theta}+qe^{-2i\theta}+q^{\frac{1}{2}}e^{i\theta}\colon \theta\in\mathbb{R}\}$$ is a hypocycloid with three cusps $(q^{\frac{3}{2}}+q+q^{\frac{1}{2}})e^{\frac{2k\pi i}{3}}$ for $k=0,1,2$ and
$$\Sigma_2=\{q(s_1+s_2+s_3)\in\mathbb{C}\colon s_1s_2s_3=1\textrm{ and }|s_1|=|s_2|=|s_3|=1\}$$ is a hypocycloid with three cusps $3qe^{\frac{2k\pi i}{3}}$ for $k=0,1,2$ and its interior (See Figure~\ref{Spectrum}).
\end{thm}

\begin{center}
\begin{figure}[h]
\begin{tikzpicture}
        \draw[dashed,->] (-5,0) -- (5,0);
        \draw[dashed,->] (0,-5) -- (0,5);
        \def\a{0.5} \def\b{1.5}
        
        \draw[fill,line width=1pt,black] plot[samples=100,domain=0:360,smooth,variable=\t] ({(\b-\a)*cos(\t)+\a*cos((\b-\a)*\t/\a},{(\b-\a)*sin(\t)-\a*sin((\b-\a)*\t/\a});
        
        \def\a{1.2} \def\b{3.6}
        \draw[line width=1pt,black] plot[samples=100,domain=0:360,smooth,variable=\t] ({(\b-\a)*cos(\t)+\a*cos((\b-\a)*\t/\a},{(\b-\a)*sin(\t)-\a*sin((\b-\a)*\t/\a});
        
        \draw[fill] (4.5,0) circle (0.05);
        \draw[fill] (120:4.5) circle (0.05);
        \draw[fill] (240:4.5) circle (0.05);
        
        \node at (1.5,1.5) {$3q$};
        \draw[->] (1.5,1.3) -- (1.5,0.1);
        \node at (3.6,1.2) {$\sqrt{q}(q+\sqrt{q}+1)$};
        \draw[->] (3.6,1) -- (3.6,0.1);
        \node at (4.5,-1) {$q^2+q+1$};
        \draw[->] (4.5,-0.8) -- (4.5,-0.1);
        
\end{tikzpicture}
\caption{$\Sigma_0\cup\Sigma_1\cup\Sigma_2$ in complex plane}\label{Spectrum}
\end{figure}
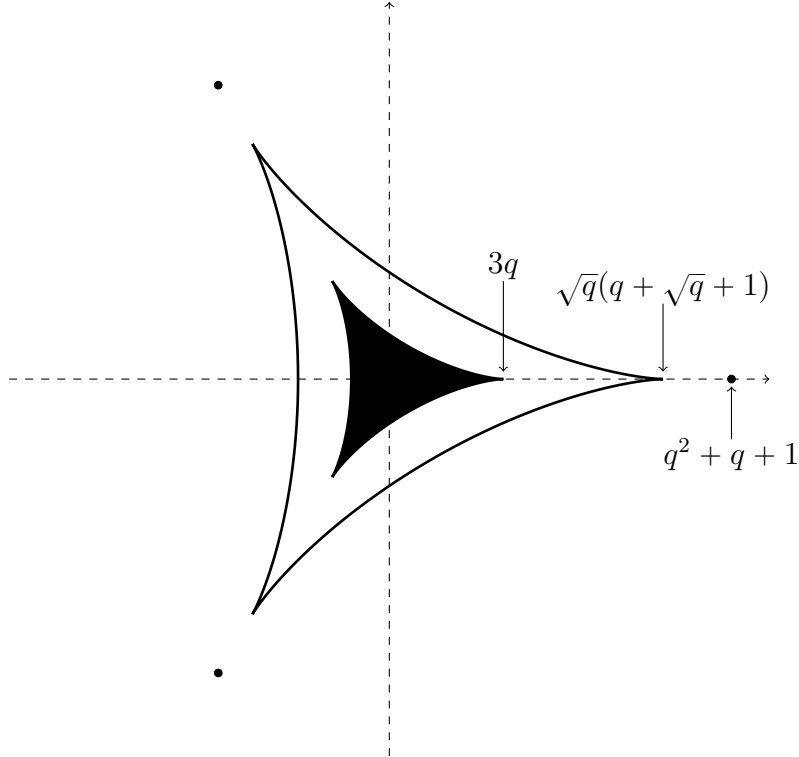
\end{center}

Since $A_w^+$ and $A_w^-$ are normal and commute with each other, if $\lambda^+$ is a spectrum of $A^+_w$, then $\lambda^-=\overline{\lambda^+}$ is also a spectrum of $A^-_w$ and vice versa. Thus, the automorphic spectrum of $A^\pm_w$ is defined as the vectors $(\lambda^+,\lambda^-)$ in $\mathbb{C}^2$ for which there exists a sequence of unit vectors $f_n\in L^2_w(\Gamma\backslash \mathcal{B}(G))$ such that 
$$\lim_{n\to\infty} (A_w^\pm f_n-\lambda^\pm f_n)=0.$$  Since $\mathcal{S}^3=\Sigma_2$ and the points in $\Sigma_1$ are also in the automorphic spectra of $A^\pm_w$, we obtain the following corollary.
\begin{coro}
$PGL(3,\mathbb{F}_q[t])\backslash\mathcal{B}(G)$ is not a Ramanujan complex.
\end{coro}

\begin{remark}
If $\Gamma$ acts properly and \emph{cocompactly} on an $\widetilde{A}_2$ building $\mathcal{B}$, then in general by \cite{CSZ} and \cite{CMS} the $L^2$-spectrum of the adjacency operator $A^+$ consists of three points $\Sigma_0$ together with a subset of the region bounded by $\Sigma_1$. By definition, the quotient complex $\Gamma\backslash\mathcal{B}$ is Ramanujan if every non-trivial $L^2$-spectrum is contained in $\Sigma_2(=\mathcal{S}^3)$.
\end{remark}

This article is organized as follows. In Section~\ref{sec:2}, we review the definition of the Bruhat-Tits building associated to the group $PGL_d$ over a non-Archimedean local field and compact Ramanujan complexes in their combinatorial and representation-theoretic forms. In Section~\ref{sec:3}, we present the structure of the non-uniform quotient $PGL(3,\mathbb{F}_q[t])\backslash PGL(3,\mathbb{F}_q(\!(t^{-1})\!))$. Colored adjacency operator with weights, the natural operator on the quotient space, is discussed in Section~\ref{sec:4}. In Section~\ref{sec:5} and \ref{sec:6}, we explore the simultaneous eigenfunctions and the spectrum of the weighted adjacency operators.

\section{Building and compact Ramanujan complex}\label{sec:2}

Let $F$ be a non-Archimedean local field with a discrete valuation $\nu$ and $\mathcal{O}$ be the valuation ring of $F$. Let $\pi$ be the uniformizer of $\mathcal{O}$ for which $\pi\mathcal{O}$ is the unique maximal ideal of $\mathcal{O}$. Let $G$ be the projective general linear group $$PGL(d,F)=GL(d,F)/\{\lambda I\colon\lambda\in F\}$$ and let $W$ be the image of the map from $GL(d,\mathcal{O})$ to $PGL(d,F)$ defined by 
$$g\rightarrow g\{\lambda I\colon\lambda\in F\}.$$ In this section, we review the affine building of type $\widetilde{A}_{d-1}$, colored adjacency operators and compact Ramanujan complexes arising as a quotient of the building for $PGL(d,F)$. For details, we refer to \cite{LSV1}.

The \textit{Bruhat-Tits building} $\mathcal{B}(G)$ associated with $G$ is the $(d-1)$-dimensinal contractible simplicial complex defined as follows. We say two $\mathcal{O}$-lattices $L$ and $L'$ of rank $d$ are in the same equivalence class if $L=sL'$ for some $s\in F^\times$. The set $\mathcal{B}(G)^0$ of vertices of $\mathcal{B}(G)$ is the set of the equivalence classes $[L]$. For given $k$-vertices $[L_1],[L_2],\cdots,[L_k]$, they form a $k$-dimensional simplex in $\mathcal{B}(G)$ if 
\begin{equation}\label{eq:1.1}
\pi L_1'\subset L_k'\subset L_{k-1}'\subset \cdots L_2'\subset L_1'
\end{equation}
for some $L_i'\in [L_i].$ In general, we denote by $\mathcal{B}(G)^k$ the $k$-skeleton of $\mathcal{B}(G)$.

Let $\mathcal{O}^d$ be the standard $\mathcal{O}$-lattice. The action of a matrix $M$ in $G$ on the set of $\mathcal{O}$-lattices transfers the standard one $\mathcal{O}^d$ to the $\mathcal{O}$-lattice of which basis consists of the column vectors $M$ and every scalar matrix $\lambda I$ preserves every equivalence class. Thus, the group $G$ acts transitively on $\mathcal{B}(G)^0.$ Since the action of $G$ on $\mathcal{B}(G)^0$ defined by left multiplication satisfies the relation \eqref{eq:1.1}, it follows that $G$ acts isometrically on the quotient space. Since the group $W$ is the stabilizer of the vertex $[\mathcal{O}^d]$, the set of vertices of Bruhat-Tits building associated to $G$ is identified with the quotient space $G/W.$

The \textit{color} $\tau:\mathcal{B}(G)^0\rightarrow \mathbb{Z}/d\mathbb{Z}$ is defined by 
$$\tau([L]):=\log_q [\mathcal{O}^d:\pi^i L],$$
for a sufficiently large positive integer $i$ with $\pi^i L\subset \mathcal{O}^d.$
Since $[\pi^iL:\pi^{i+1}L]=d$, the color $\tau([L])$ is independent of the choice of the lattice in $[L]$ and hence is well-defined.

Let 
$L^2(\mathcal{B}(G))$ be the space of functions $f\colon\mathcal{B}(G)^0\rightarrow \mathbb{C}$ satisfying $$\sum_{x\in \mathcal{B}(G)^0}|f(x)|^2<\infty.$$ 

The \textit{colored adjacency operator} $A_i:L^2(\mathcal{B}(G))\rightarrow L^2(\mathcal{B}(G))$ is defined for $f\in L^2(\mathcal{B}(G))$ by
\begin{equation}\label{eq:1.2}
 A_if(x)=\sum_{\substack{ y\sim x\\ \tau(y)=\tau(x)+i}}f(y),
 \end{equation}
 where $y\sim x$ implies that there is an edge between $y$ and $x$ in $\mathcal{B}(G)$. These operators are bounded and commutative. Moreover, since $A_i^*=A_{k-i}$, the operators are normal. Let $\mathcal{S}^d\subset\mathbb{C}^{d-1}$ be the simultaneous spectrum of $(A_1,A_2,\ldots,A_{d-1})$ acting on $L^2(\mathcal{B}(G))$. More precisely, the spectrum $\mathcal{S}^d$ is the subset of $d$-tuples $(\lambda_1,\cdots,\lambda_{d-1})$ in $\mathbb{C}^{d-1}$ such that there exists a sequence of functions $f_n\in L^2(\mathcal{B}(G))$ with $\|f_n\|_2=1$ satisfying, for any  $i=1,\ldots, d-1,$
 $$\lim_{n\to\infty}\|A_if_n-\lambda_if_n\|_2=0.$$

Although the non-uniform lattice $PGL(3,\mathbb{F}_q[t])$ of $PGL(3,\mathbb{F}_q(\!(t^{-1})\!))$ is the main topic of this paper, in order to motivate the reader, we review the definition of compact Ramanujan complexes arising as a quotient of $PGL(d,F)$. Let $\Gamma$ be a torsion-free cocompact discrete subgroup of $G$. Then $\Gamma$ acts on $\mathcal{B}(G)^0=G/W$ by left translation, and $\Gamma\backslash \mathcal{B}(G)$ is a finite complex. The color function defined on $\mathcal{B}(G)^0$ may not be preserved by $\Gamma$. However, the colors defined on the set $\mathcal{B}(G)^1$ of edges by 
$$\tau(x,y)=\tau(x)-\tau(y)\quad(\textrm{mod }d)$$ are preserved by $\Gamma$. 

Let $L^2(\Gamma\backslash \mathcal{B}(G))$ be the space of functions $f$ on $\Gamma\backslash\mathcal{B}(G)^0$ with $\|f\|_{2}<+\infty$. Here, $\|f\|_2=\sum_{x\in\Gamma\backslash\mathcal{B}(G)^0}|f(x)|^2$. Using the equation \eqref{eq:1.2}, it follows that operators $A_i$ induce colored adjacency operators
 $$A_{X,i}:L^2(\Gamma\backslash \mathcal{B}(G))\rightarrow L^2(\Gamma\backslash \mathcal{B}(G))$$
 on $\Gamma\backslash \mathcal{B}(G)$.
 The eigenfunction $f$ of $A_{X,i}$ on $L^2(\Gamma \backslash \mathcal{B}(G))$ is called \textit{trivial} if the function $f$ is of the form
$$f([L])=\xi^{\tau([L])},$$
for some $d$-th root $\xi$ of unity, i.e., $\xi^d=1.$ 
\begin{defn}\label{def:1.1}The quotient complex $\Gamma \backslash \mathcal{B}(G)$ is called \textit{Ramanujan} if for every non-trivial simultaneous eigenfunction $f$ of the colored adjacency operators $A_{X,i}$ acting on $L^2(\Gamma\backslash\mathcal{B}(G))$, the simultaneous eigenvalue $(\lambda_1,\ldots,\lambda_{d-1})$ for $f$ is contained in $\mathcal{S}^d$.
\end{defn}


We remark that the authors in \cite{LSV1} proved that if $d$ is a prime and $\Gamma$ is an arithmetic uniform lattice of inner type, then $\Gamma \backslash \mathcal{B}(G)$ is Ramanujan.

\section{Non-uniform arithmetic quotient $PGL(3,\mathbb{F}_q[t])\backslash PGL(3,\mathbb{F}_q(\!(t^{-1})\!))$}\label{sec:3}

Let $\mathbb{F}_q$ be the finite field of order $q$ and let $\mathbb{F}_q(t)$ be the field of rational functions over $\mathbb{F}_q$. The absolute value $\|\cdot\|$ of $\mathbb{F}_q(t)$ is defined for any $f\in \mathbb{F}_q(t)$, by 
$$\|{f}\|:=q^{\deg (g)-\deg (h)},$$
where $g,h$ are polynomial over $\mathbb{F}_q$ satisfying $f=\frac{g}{h}$.
The completion of $\mathbb{F}_q(t)$ with respect to $\|\cdot\|$, the field of formal Laurent series in $t^{-1}$, is denoted by $\mathbb{F}_q(\!(t^{-1})\!)$, i.e.,
$$\mathbb{F}_q(\!(t^{-1})\!):=\left\{\sum_{n=-N}^\infty a_nt^{-n}:N\in \mathbb{Z}, a_n\in \mathbb{F}_q\right\}.$$
The valuation ring $\mathcal{O}$ is the subring of power series 
$$\mathbb{F}_q\mathbb{[\![}t^{-1}]\!]:=\left\{\sum_{n=0}^\infty a_nt^{-n}: a_n\in \mathbb{F}_q\right\}.$$

Let $G$ be the group $PGL(3,\mathbb{F}_q(\!(t^{-1})\!))$, $\Gamma$ be its non-uniform lattice $PGL(3,\mathbb{F}_q[t])$ and $W=PGL(3,\mathcal{O})$. In this section and throughout, we focus on the building $\mathcal{B}(G)$ for $G$ and the quotient $\Gamma\backslash \mathcal{B}(G)$ for $\Gamma=PGL(3,\mathbb{F}_q[t])$. In general, given a simple and simply-connected Chevalley group scheme $H$ defined over $\mathbb{Z}$, the action of $H(\mathbb{F}_q[t])$ on $\mathcal{B}(H)$ is described in \cite{Soule}.

 We recall that the set of vertices $\mathcal{B}(G)^0$ is identified with the coset space $G/W$. Since the right multiplication by a matrix in $W$ is considered as the elementary column $\mathcal{O}$-operation, we may find a representative $[A]=AW$ of vertex in $\mathcal{B}(G)$ with the following conditions:
\begin{itemize}
\item The matrix $A$ is upper diagonal. 
\item The diagonal entries of $A$ are of the form $t^{n}$ where $n\in \mathbb{Z}$.
\item The $(i,j)$-entry $a_{ij}$ of $A$ is contained in $t^{n+1}\mathbb{F}_q[t]$ whenever $a_{ii}=t^n$.
\item If $a_{ii}=a_{jj}$, then $a_{ij}=a_{ji}=0$.
\end{itemize}

Similarly, the left action of $\Gamma$ is considered as the elementary row operation. The following lemmata describe the fundamental domain for $\Gamma$-action on $\mathcal{B}(G)^0$. These actually follow from the general theorem in \cite{Soule}, but to be self-contained we give the elementary proof of the statements.

\begin{lem}\label{lem:PGL2}
Given every $g\in PGL(2,\mathbb{F}_q(\!(t^{-1})\!))$, there exists a unique non-negative integer $n$ such that
$$g=\gamma\begin{pmatrix} t^n & 0 \\ 0 & 1 \end{pmatrix}w$$
for some $\gamma\in PGL(2,\mathbb{F}_q[t])$ and $w\in PGL(2,\mathcal{O})$.
\end{lem}
\begin{proof}
Let $\Gamma_2=PGL(2,\mathbb{F}_q[t])$, $W_2=PGL(2,\mathcal{O})$ and $[\alpha]$ be the polynomial part of $\alpha$. Let $T\colon \mathbb{F}_q(\!(t^{-1})\!)\to\mathbb{F}_q(\!(t^{-1})\!)$ be the mapping given by $T(\alpha)=\frac{1}{\alpha-[\alpha]}$. From the above observation, every $g\in G$ has a representative 
$$\begin{pmatrix} t^m & \alpha \\ 0 & 1 \end{pmatrix}$$
for $m\in\mathbb{Z}$ and $\alpha\in t^{m+1}\mathbb{F}_q[t]$. If $m\ge 0$, then $\alpha\in\mathbb{F}_q[t]$ and hence
\begin{align*}
\Gamma_2\begin{pmatrix} t^m & \alpha \\ 0 & 1 \end{pmatrix}W_2=\Gamma_2\begin{pmatrix} t^m & 0 \\ 0 & 1 \end{pmatrix}W_2
\end{align*}
since $$\begin{pmatrix} t^m & \alpha \\ 0 & 1 \end{pmatrix}=\begin{pmatrix} 1 & \alpha \\ 0 & 1 \end{pmatrix}\begin{pmatrix} t^m & 0 \\ 0 & 1 \end{pmatrix}.$$
If $m<0$, then we may assume that $\alpha=a_{-m+1}t^{-m+1}+\cdots+a_{-1}t^{-1}$ and we have
\begin{align*}
&\Gamma_2\begin{pmatrix} t^m & \alpha \\ 0 & 1 \end{pmatrix}W_2=\Gamma_2\begin{pmatrix} t^m & \alpha-[\alpha] \\ 0 & 1 \end{pmatrix}W_2=\Gamma_2\begin{pmatrix} 0 & \alpha-[\alpha] \\ -\frac{t^m}{\alpha-[\alpha]} & 1 \end{pmatrix}W_2\\
=\,&\Gamma_2\begin{pmatrix} -t^mT(\alpha) & 1 \\ 0 & \alpha-[\alpha] \end{pmatrix}W_2=\Gamma_2\begin{pmatrix} -t^mT(\alpha)^2 & T(\alpha) \\ 0 & 1 \end{pmatrix}W_2.
\end{align*}
Since $\alpha$ is rational, we have $T^{k}(\alpha)=0$ for large enough $k$ which implies that the above reduction eventually stops. After exchanging row and column if necessary, we get
$$\Gamma\begin{pmatrix} t^m & \alpha \\ 0 & 1 \end{pmatrix}W=\Gamma\begin{pmatrix}t^n & 0 \\ 0 & 1 \end{pmatrix}W$$
for some $n\ge 0$.
\end{proof}

\begin{lem}
Given every $g\in G$, there exists a unique pair of non-negative integers $(m,n)$ with $0\le n\le m$ such that
$$g\in\Gamma\begin{pmatrix}t^m & 0 & 0 \\ 0 & t^n & 0 \\ 0 & 0 & 1\end{pmatrix} W.$$
\end{lem}
\begin{proof}
 From the above observation, every $g\in G$ may be written by 
$$\begin{pmatrix}t^{m_0} & a_{12} & a_{13} \\ 0 & t^{n_0} & a_{23} \\ 0 & 0 & t^{\ell_0}\end{pmatrix}w$$
for some $a_{ij}\in\mathbb{F}_q(\!(t^{-1})\!)$, matrix $w\in W$ and non-negative integers $m_0, n_0$ and $\ell_0$. If $m_0< n_0$, then multiplying suitable $\gamma$ on the left and $w$ on the right so that we may assume $a_{12}$ is zero. Similary, if $n_0< \ell_0$, then we may assume $a_{23}$ is zero. If either $a_{12}$ or $a_{23}$ is zero, then it reduces to the case where it is not so it suffices to consider the case $m_0\ge n_0\ge \ell_0$. In this case, applying the Lemma~\ref{lem:PGL2} to the upper-left $2\times 2$ block of the matrix, we may find $\gamma_1\in \Gamma$ such that
$$\begin{pmatrix}t^{m_0} & a_{12} & a_{13} \\ 0 & t^{n_0} & a_{23} \\ 0 & 0 & t^{\ell_0}\end{pmatrix}=\gamma_1\begin{pmatrix}t^{m_1} & 0 & \star \\ 0 & t^{k_1} & \star \\ 0 & 0 & t^{\ell_0}\end{pmatrix}w_1$$
with $m_0\ge m_1\ge k_1\ge n_0$.
Now, applying the Lemma~\ref{lem:PGL2} (and the proof) to the lower-right $2\times 2$ block, we may find $\gamma_1'\in \Gamma$ and $w_1'\in W$ such that
$$\begin{pmatrix}t^{m_1} & 0 & \star \\ 0 & t^{k_1} & \star \\ 0 & 0 & t^{\ell_0}\end{pmatrix}=\gamma_1'\begin{pmatrix}t^{m_1} & \star & \star \\ 0 & t^{n_1} & 0 \\ 0 & 0 & t^{\ell_1}\end{pmatrix}w_1'$$
with $k_1\ge n_1\ge \ell_1\ge \ell _0$. While the reducing process $$(m_i,n_i,\ell_i)\rightarrow (m_{i+1},n_{i+1},\ell_{i+1})$$ goes on, both $(1,2)$-entry and $(2,3)$-entry become eventually zero since $m_{i}> m_{i+1}\ge \ell_{i+1}>\ell_{i}$ unless both entries are zero. Continuing this process until it eventually stops, we may obtain $\gamma,\gamma'\in \Gamma$ and $w,w'\in W$ for which
$$\begin{pmatrix}t^{m_0} & 0 & \star \\ 0 & t^{n_0} & \star \\ 0 & 0 & t^{\ell_0}\end{pmatrix}=\gamma'\begin{pmatrix}t^{m_t} & 0 & \star \\ 0 & t^{n_t} & 0 \\ 0 & 0 & t^{\ell_t}\end{pmatrix}w'=\gamma\begin{pmatrix}t^{m} & 0 & 0 \\ 0 & t^{n} & 0 \\ 0 & 0 & 1\end{pmatrix}w$$
with $m\ge n\ge 0$.
\end{proof}

For $m,n$ with $m\ge n\ge 0$, let $v_{m,n}$ be the vertex of the quotient complex $\Gamma\backslash \mathcal{B}(G)$ corresponds to
$$\Gamma \begin{pmatrix} t^m & 0 & 0 \\ 0 & t^n & 0 \\ 0 & 0 & 1\end{pmatrix}W.$$ 
The definition of the building $\mathcal{B}(G)$ implies that there is an edge between two vertices $v_{m,n}$ and $v_{m',n'}$ if and only if the following hold:
\begin{displaymath}
\begin{cases} (m',n')\in\{(m\pm1,n),(m,n\pm1),(m\pm1,n\pm1)\} &\text{ if }m>n>0\\
(m',n')\in \{(m\pm1,n),(m,n+1),(m+1,n+1)\} &\text{ if }m>n=0\\
(m',n')\in \{(m+1,n),(m,n-1),(m\pm1,n\pm1)\} &\text{ if }m=n>0\\
(m',n')\in \{(1,0),(1,1)\}&\text{ if }m=n=0.
\end{cases}
\end{displaymath}
Combining above facts, the quotient complex $\Gamma\backslash\mathcal{B}(G)$ is described as Figure~\ref{figure}.
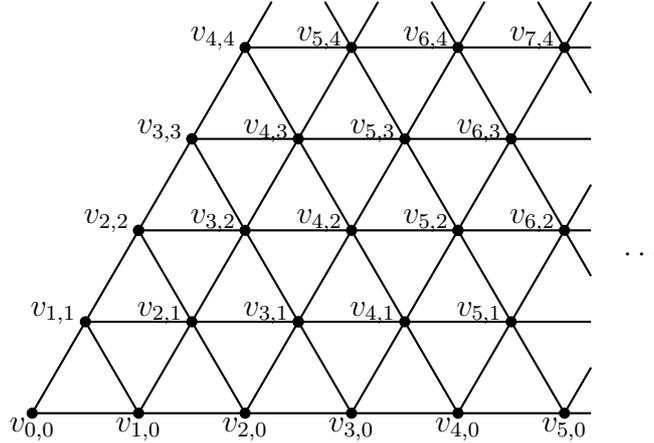
\begin{figure}[h]
\begin{center}
\begin{tikzpicture}[scale=0.7]
  \draw [thick](-5,0) -- (5.5,0);   \draw [thick](-0.5,+{sqrt(60.75)}) -- (-5,0);\draw[thick](-4,+{sqrt(3)}) -- (-3,0);\draw[thick] (-3,+{sqrt(12)}) -- (-1,0);\draw [thick](-2,{sqrt(27)}) -- (1,0);\draw[thick] (-1,{sqrt(48)})--(3,0);\draw[thick] (-4,+{sqrt(3})--(5.5,+{sqrt(3)});\draw[thick](-3,+{sqrt(12)})--(5.5,+{sqrt(12)});\draw[thick](-2,+{sqrt(27)})--(5.5,+{sqrt(27)});\draw[thick](-1,+{sqrt(48)})--(-1,+{sqrt(48)});\draw[thick](1.5,+{sqrt(60.75)})--(-3,0);\draw[thick](3.5,+{sqrt(60.75)})--(-1,0);\draw[thick](5.5,+{sqrt(60.75)})--(1,0);\draw[thick](5.5,+{sqrt(75)}/2)--(3,0);\draw[thick](5,0)--(0.5,+{sqrt(60.75)});\draw[thick] (-1,+{sqrt(48)})--(5.5,+{sqrt(48)});\draw[thick](5.5,+{sqrt(27)}/2)--(2.5,+{sqrt(60.75)});\draw[thick] (4.5,+{sqrt(60.75)})--(5.5,+{sqrt(147)}/2);\draw[thick] (5,0)--(5.5,+{sqrt(3)}/2);
\node at (-5,-0.3) {$v_{0,0}$};\node at (-3,-0.3) {$v_{1,0}$};\node at (-1,-0.3) {$v_{2,0}$};\node at (1,-0.3) {$v_{3,0}$};\node at(3,-0.3) {$v_{4,0}$};\node at (5,-0.3) {$v_{5,0}$};\node at (6.5,3) {$\cdots$};\node at (-4.6,1.9) {$v_{1,1}$};\node at (-2.6,1.9) {$v_{2,1}$};\node at (-0.6,1.9) {$v_{3,1}$};\node at (1.4,1.9) {$v_{4,1}$};\node at (3.4,1.9) {$v_{5,1}$};\node at (-3.6,3.65) {$v_{2,2}$};\node at (-1.6,3.65) {$v_{3,2}$};\node at (0.4,3.65) {$v_{4,2}$};\node at(2.4,3.65) {$v_{5,2}$};\node at (4.4,3.65) {$v_{6,2}$};\node at (-2.6,5.35) {$v_{3,3}$};\node at (-0.6,5.35) {$v_{4,3}$};\node at (1.4,5.35) {$v_{5,3}$};\node at (3.4, 5.35) {$v_{6,3}$};\node at (-1.6, 7.1) {$v_{4,4}$};\node at (0.4, 7.1) {$v_{5,4}$};\node at (2.4, 7.1) {$v_{6,4}$};\node at (4.4, 7.1) {$v_{7,4}$};
\fill (-5,0)    circle (3pt); \fill (-3,0)    circle (3pt); \fill (-1,-0)    circle (3pt);\fill (1,-0)    circle (3pt); \fill (3,0) circle (3pt);\fill(5,0) circle (3pt);\fill (-4,+{sqrt(3)})    circle (3pt); \fill (-3,+{sqrt(12)})    circle (3pt); \fill (-2,{sqrt(27)})    circle (3pt);\fill (-1,+{sqrt(48)})    circle (3pt); \fill (-2,+{sqrt(3)}) circle (3pt);\fill(-0,+{sqrt(3)}) circle (3pt);\fill (0,+{sqrt(3)}) circle (3pt);\fill(2,{sqrt(3)}) circle (3pt);\fill (4,+{sqrt(3)}) circle (3pt);\fill (-3,{sqrt(12)}) circle (3pt); \fill (-1,{sqrt(12)}) circle (3pt);\fill (1,{sqrt(12)}) circle (3pt);\fill(3,{sqrt(12)}) circle (3pt);\fill (3,{sqrt(12)}) circle (3pt);\fill(5,{sqrt(12)}) circle (3pt);\fill (-2,{sqrt(27)}) circle (3pt);\fill (0,{sqrt(27)}) circle (3pt);\fill (2,{sqrt(27)}) circle (3pt); \fill (4,{sqrt(27)}) circle (3pt);\fill(1,{sqrt(48)}) circle (3pt); \fill(3,{sqrt(48)}) circle (3pt);\fill (5,{sqrt(48)}) circle (3pt);
\end{tikzpicture}
\end{center}
\caption{The fundamental domain for $\Gamma\backslash\mathcal{B}(G)$}\label{figure}
\end{figure}
\section{Colored adjacency operator with weights}\label{sec:4}
Generalizing the idea of \cite{M}, the colored adjacency operators with weights are introduced in \cite{S}. In this section, we study the combinatorics of the natural weighted adjacency operators $A_w^+$ and $A_w^-$. Since $\Gamma$ acts on $\mathcal{B}(G)$ with torsions, the induced operators on the quotient from the colored adjacency operators on $L^2(\mathcal{B}(G))$ come with weights envolving the cardinality of the stabilizer of the vertices and edges.

Given a simplicial complex $X$, let $V=V(X)$ and $E=E(X)$ be the set of vertices and edges of $X$, respectively. We denote by $(u,v)$ the edge with vertices $u$ and $v$. 
\begin{defn}\label{def:3.1}
The \textit{weight function} $w:V\cup E\rightarrow (0,1]$ is a function such that $w(x)\leq w(e)$ holds for every $v\in V$ and $e\in E$ with $v\in e$. The function $\theta(u,v)=\frac{w(e)}{w(u)}$ is called the \emph{entering degree} of $e$ to $u$, where $e=(u,v)$. The \textit{in-degree} of a vertex $u$ is $\text{in-degree}(u):=\underset{(u,v)\in E}\sum \theta(u,v).$ A simplicial complex $X$ with weight function $w$ is $k$-regular if $\text{in-degree}(u)=k$ for any $u\in V$. \end{defn}
The weight function $w$ allows to define a measure $\mu$ on $V$, for any $S\subset V$, by $\mu(S)=\underset{v\in S}\sum w(v).$ Using the measure $\mu$, we define the $L^2$-norm of a function $f$ by
$$\|f\|_2:=\biggl(\sum_{u\in V}|f(u)|^2w(u)\biggr)^{1/2}$$
and let $L^2_w(X)$ be the space of functions $f\colon V(X)\to\mathbb{C}$ with $\|f\|_2<\infty$.

The \textit{weighted adjacency operators} $A^+_w$ and $A^-_w$ on $L^2_w(\Gamma\backslash\mathcal{B}(G))$ is defined for any $f\in L^2_w(\Gamma\backslash\mathcal{B}(G))$ by
$$A_w^\pm f(u):=\sum_{\substack{(u,v)\in E\\ \tau(v)=\tau(u)\pm1}}\frac{w(u,v)}{w(v)}f(v).$$
The operators $A_w^\pm$ satisfies that $(A_w^+)^*=A_w^-$ (see \cite{S}, Remark 4.1).

 Let $\Gamma_{m,n}$ be the stabilizer of the vertex $$x_{m,n}W=\begin{pmatrix} t^m & 0 & 0 \\ 0 & t^n & 0 \\ 0 & 0 & 1\end{pmatrix}W$$ in $\mathcal{B}(G)$. Let $w$ be the weight function defined by 
 \begin{displaymath}
 w(v_{m,n}):=\frac{q^3(q+1)(q-1)^2}{ |\Gamma_{m,n}|} \text{ and } w(e):=\frac{q^3(q+1)(q-1)^2}{|\Gamma_{v_{m,n}}\cap\Gamma_{v_{m',n'}}|},
 \end{displaymath}
where $e=(v_{m,n},v_{m',n'})$ and $|S|$ denotes the cardinality of the set $S$. With respect to these weights $w$, the colored adjacency operators on $\mathcal{B}(G)$ induce the weighted adjacency operators $A_w^+$ and $A_w^-$ on the quotient $L^2_w(\Gamma\backslash \mathcal{B}(G))$.
The rest part of this section is devoted to compute the values of the weight function $w$.
 \begin{prop}\label{prop:3.2}
Let $N_{m,n}=|\Gamma_{m,n}|$. We have
\begin{equation}\label{eq:3.1}
N_{m,n}=
\begin{cases}
q^3(q+1)(q^2+q+1)(q-1)^2 &\text{ if }\, m=n=0\\
q^{2m+3}(q+1)(q-1)^2 &\text{ if }\,m>n=0\\
q^{2m+3}(q-1)^2 &\text{ if }\, m>n>0\\
q^{2m+3}(q+1)(q-1)^2&\text{ if }\, m=n>0.
\end{cases}
\end{equation}
\end{prop}
\begin{proof}Since $\gamma x_{m,n}W=x_{m,n}W$ if and only if $\gamma\in\Gamma_{m,n},$ we have
$$\Gamma_{m,n}=\{\gamma\in \Gamma:x_{m,n}^{-1}\gamma x_{m,n}\in W\}.$$
Thus, $\Gamma_{0,0}=\Gamma\cap PGL(3,\mathcal{O})=PGL(3,\mathbb{F}_q)$ and
\begin{equation}\label{eq:3.2}
\begin{split}
\Gamma_{m,0}&=\biggl\{\begin{pmatrix} a_{11}&a_{12}&a_{13}\\0&a_{22}&a_{23}\\0&a_{32}&a_{33}\end{pmatrix}\in GL(3,\mathbb{F}_q(\!(t^{-1})\!)):\substack{a_{11},a_{22},a_{23},a_{32},a_{33}\in \mathbb{F}_q, \\a_{12},a_{13}\in P^m(t)}\biggr\}/\{\lambda I:\lambda\in \mathbb{F}_q^\times\}\\
\Gamma_{m,n}&=\biggl\{\begin{pmatrix} a_{11}&a_{12}&a_{13}\\0&a_{22}&a_{23}\\0&0&a_{33}\end{pmatrix}\in GL(3,\mathbb{F}_q(\!(t^{-1})\!)):\substack{a_{11},a_{22},a_{33}\in \mathbb{F}_q, a_{12}\in P^{m-n}(t),\\a_{13}\in P^m(t),a_{23}\in P^n(t)}\biggr\}/\{\lambda I:\lambda\in \mathbb{F}_q^\times\}\\
\Gamma_{m,m}&=\biggl\{\begin{pmatrix} a_{11}&a_{12}&a_{13}\\a_{21}&a_{22}&a_{23}\\0&0&a_{33}\end{pmatrix}\in GL(3,\mathbb{F}_q(\!(t^{-1})\!)):\substack{a_{11},a_{21},a_{22},a_{23},a_{33}\in \mathbb{F}_q, \\a_{13},a_{23}\in P^m(t)}\biggr\}/\{\lambda I:\lambda\in \mathbb{F}_q^\times\},
\end{split}
\end{equation}
for $m>n>0$, where $P^n(t)$ is the space of polynomials of degree less than or equal to $n$.
Since $$|GL(3,\mathbb{F}_q)|=(q^3-1)(q^3-q)(q^3-q^2),$$
it follows that $N_{0,0}=|PGL(3,\mathbb{F}_q)|=q^3(q+1)(q^2+q+1)(q-1)^2.$ 

For any $\gamma\in \Gamma_{m,0}$, $a_{11}\neq 0$, two vectors $(a_{22},a_{32})$ and $(a_{32},a_{33})$ are linearly independent. This shows that 
$$N_{m,0}=\frac{(q-1)(q^2-1)(q^2-q)q^{2m+2}}{q-1}=q^{2m+3}(q+1)(q-1)^2.$$
Simliarly, we also have $N_{m,m}=q^{2m+3}(q+1)(q-1)^2.$

Since $a_{ii}\neq 0$ for any $\gamma\in \Gamma_{m,n}$ with $m>n>0$, it follows that
$$N_{m,n}=\frac{(q-1)^3q^{m-n+1}q^{m+1}q^{n+1}}{q-1}=q^{2m+3}(q-1)^2$$
which completes the proof of the proposition.
\end{proof}
Proposition \ref{prop:3.2} implies that 
\begin{equation}\label{eq:3.3}
w(v_{m,n})=
\begin{cases}\vspace{0.5em}
\dfrac{1}{q^2+q+1} &\text{ if }\, m=n=0\\\vspace{1em}
\dfrac{1}{q^{2m}}&\text{ if }\,m>n=0\\\vspace{0.5em}
\dfrac{q+1}{q^{2m}} &\text{ if }\, m>n>0\\
\dfrac{1}{q^{2m}}&\text{ if }\, m=n>0.
\end{cases}
\end{equation}
It follows from \eqref{eq:3.2} that for any $m,n$ with $m>n>0,$
\begin{equation}\label{eq:3.4}
\begin{split}
&\Gamma_{m,0}\cap\Gamma_{m+1,0}=\Gamma_{m,0}\\
&\Gamma_{m,m}\cap\Gamma_{m+1,m+1}=\Gamma_{m,m} \text{ and }\\
&\Gamma_{m,n}\cap\Gamma_{m+1,n}=\Gamma_{m,n}\\
&\Gamma_{m,n}\cap\Gamma_{m+1,n+1}=\Gamma_{m,n}.
\end{split}
\end{equation}
The last equation of \eqref{eq:3.4} holds when $n=0$.
The following proposition gives the value of the weight function $w$ on edges.
\begin{prop} Under the above notation, we have
\begin{equation}\label{eq:3.5}
\begin{split}
&|\Gamma_{0,0}\cap\Gamma_{1,0}|=q^3(q+1)(q-1)^2\\
&|\Gamma_{0,0}\cap \Gamma_{1,1}|=q^3(q+1)(q-1)^2\\
&|\Gamma_{1,0}\cap \Gamma_{1,1}|=q^4(q-1)^2\\
&|\Gamma_{m,n-1}\cap\Gamma_{m,n}|=(q-1)^2q^{2m+3}.
\end{split}
\end{equation}
\end{prop}
\begin{proof}
Applying \eqref{eq:3.2},
\begin{equation}\label{3.5}
\nonumber\begin{split}
&\Gamma_{0,0}\cap\Gamma_{1,0}=\biggl\{\begin{pmatrix} a_{11}&a_{12}&a_{13}\\0&a_{22}&a_{23}\\0&a_{32}&a_{33}\end{pmatrix}:a_{11},a_{12},a_{13},a_{22},a_{23},a_{32}a_{33}\in \mathbb{F}_q\biggr\}/\{\lambda I:\lambda\in \mathbb{F}_q^\times\}\\
&\Gamma_{0,0}\cap\Gamma_{1,0}=\biggl\{\begin{pmatrix} a_{11}&a_{12}&a_{13}\\a_{21}&a_{22}&a_{23}\\0&0&a_{33}\end{pmatrix}:a_{11},a_{12},a_{13},a_{21}a_{22},a_{23},a_{33}\in \mathbb{F}_q\biggr\}/\{\lambda I:\lambda\in \mathbb{F}_q^\times\}\\
&\Gamma_{1,0}\cap\Gamma_{1,1}=\biggl\{\begin{pmatrix} a_{11}&a_{12}&a_{13}\\0&a_{22}&a_{23}\\0&0&a_{33}\end{pmatrix}:\substack{a_{11},a_{12},a_{22},a_{23},a_{33}\in \mathbb{F}_q\\ a_{13}\in P^1(t)}\biggr\}/\{\lambda I:\lambda\in \mathbb{F}_q^\times\}
\end{split}
\end{equation}
and 
\begin{equation}\nonumber
\Gamma_{m,n-1}\cap\Gamma_{m,n}=\biggl\{\begin{pmatrix} a_{11}&a_{12}&a_{13}\\0&a_{22}&a_{23}\\0&0&a_{33}\end{pmatrix}:\substack{a_{11},a_{22},a_{33}\in \mathbb{F}_q,a_{12}\in P^{m-n}(t)\\a_{13}\in P^m(t),a_{23}\in P^{n-1}(t)}\biggr\}/\{\lambda I:\lambda\in \mathbb{F}_q^\times\}
\end{equation}
Simliar to the proof of Proposition \ref{prop:3.2}, we get \eqref{eq:3.5}.
\end{proof}

\section{Eigenfunctions of the weighted adjacency operators}\label{sec:5}
Let $S=\{(s_1,s_2,s_3)\in\mathbb{C}^3\colon s_1s_2s_3=1\text{ and }\overline{s_1+s_2+s_3}=s_1^{-1}+s_2^{-1}+s_3^{-1}\}.$ 
Since $(A_w^+)^*=A_w^-$, a pair $(\lambda^+,\lambda^-)\in\mathbb{C}^2-\{(0,0)\}$ of simultaneous eigenvalues of $A_w^+$ and $A_w^-$ satisfies $\overline{\lambda^+}=\lambda^-$. We note that every $(\lambda^+,\lambda^-)$ can be described by
\begin{equation}\label{eq:4.1}
\lambda^+:={q}(s_1+s_2+s_3)\text{ and }\lambda^-:={q}\left(\frac{1}{s_1}+\frac{1}{s_2}+\frac{1}{s_3}\right).\\
\end{equation}
for some point  $(s_1,s_2,s_3)$ in $S$.
In this section, we investigate all eigenfunctions of $A_w^\pm$ in $C(\Gamma\backslash\mathcal{B}(G))$, using the parametrization \eqref{eq:4.1} and a family of recursive relations of $A_w^\pm$. Although not every eigenfunction involves automorphic spectrum, this enables us in Section~\ref{sec:6} to find approximate eigenvalues of $A_w^\pm$ on $L^2_w(\Gamma\backslash\mathcal{B}(G))$. 

\begin{lem}\label{b}For any $(s_1,s_2,s_3)\in S\backslash\{(s_1,s_2,s_3)\in S:|s_1|=|s_2|=|s_3|=1\}$, there exists a permutation $\sigma$ on $\{1,2,3\}$ such that 
\begin{equation}\label{a}
\overline{s_{\sigma(1)}}=s_{\sigma(3)}^{-1}\text{ and }|s_{\sigma(2)}|=1.
\end{equation}
\end{lem}
\begin{proof}Without loss of generality, we may assume that $|s_1|>1$ and $|s_2|<1$. Put $$s_1=ae^{i\theta_1},\,\,s_2=be^{i\theta_2}\,\,\text{ and }\,\,s_3=\frac{1}{ab}e^{-i(\theta_1+\theta_2)},$$
where $a,b>0.$ 
Since $|\overline{s_1}-s_1^{-1}+\overline{s_2}-s_2^{-1}|^2=|s_3^{-1}-\overline{s_3}|^2,$ we have 
\begin{equation}
\begin{split}
b^2(a^2-{1})^2+2ab(a^2-{1})(b^2-{1})\cos(\theta_1-\theta_2)+a^2(b^2-{1})^2&=(a^2b^2-{1})^2.\nonumber
\end{split}
\end{equation}
Using the above equation, we have $2\cos (\theta_1-\theta_2)=ab+\frac{1}{ab}$. This holds only if $ab=1$ and $\theta_1=\theta_2$. Thus it is possible to choose $\sigma$ satisfying \eqref{a}.
\end{proof}

Now we will solve the equations for eigenfunctions of $A_w^+$ and $A_w^-$. In Chapter~3 of \cite{CM} where the authors investigate the spherical functions on $\mathcal{B}(G)$, similar recurrence formulas appear with different coefficients.

\begin{prop}
Let $f$ be an eigenfunction of $A_w^+$ and $A_w^-$ with eigenvalues $\lambda^+$ and $\lambda^-$. Suppose that $s_i\neq s_j$ for any $i\neq j$. Then, $f$ is given by
$$f(v_{m,n})=\sum_{\substack{(i,j)\in \{1,2,3\}^2\\i\neq j}}B_{i,j}q^ms_i^ms_{j}^n$$
where 
$$B_{i,j}=\frac{(s_{i}-qs_{j})(s_i-qs_{k})(s_j-qs_{k})}{(s_{i}-s_j)(s_{i}-s_k)(s_{j}-s_{k})(q+1)(q^2+q+1)}.$$ 
\end{prop}

\begin{proof}
Let $f$ be an eigenfunction of $A_w^+$ and $A_w^-$ with eigenvalues $\lambda^+$ and $\lambda^-$. For convenience, suppose that $f(v_{0,0})=1$. The definition of $A^\pm_w$ together with the equations \eqref{eq:3.1}, \eqref{eq:3.4} and \eqref{eq:3.5} show that the function $f$ satisfies
\begin{eqnarray}
\label{2}A_w^+f(v_{0,0})&=&{(q^2+q+1)f(v_{1,0})}=\lambda^+\\
\label{3}A_w^-f(v_{0,0})&=&(q^2+q+1)f(v_{1,1})=\lambda^-\\
\label{4}A_w^+f(v_{1,0})&=&{f(v_{2,0})+(q^2+q)f(v_{1,1})}=\lambda^+f(v_{1,0})=(\lambda^+)^2\\
\label{5}A_w^-f(v_{1,0})&=&{q^2f(v_{0,0})+(q+1)f(v_{2,1})}=\lambda^-f(v_{1,0})=\lambda^+\lambda^-\\
\label{6}A_w^+f(v_{1,1})&=&{q^2f(v_{0,0})+(q+1)f(v_{2,1})}=\lambda^+f(v_{1,1})=\lambda^+\lambda^-\\
\label{7}A_w^-f(v_{1,1})&=&{(q^2+q)f(v_{1,0})+f(v_{2,2})}=\lambda^-f(v_{1,1})=(\lambda^-)^2.
\end{eqnarray}
For $m\geq 1$, the function $f$ have the following properties:
\begin{eqnarray}
\label{8}A_w^+f(v_{m,0})&=&{f(v_{m+1,0})+(q^2+q)f(v_{m,1})}=\lambda^+f(v_{m,0})\\
\label{9}A_w^-f(v_{m,0})&=&{q^2f(v_{m-1,0})+(q+1)f(v_{m+1,1})}=\lambda^-f(v_{m,0})\\
\label{10}A_w^+f(v_{m,m})&=&{q^2f(v_{m-1,m-1})+(q+1)f(v_{m+1,m})}=\lambda^+f(v_{m,m})\\
\label{11}A_w^-f(v_{m,m})&=&{(q^2+q)f(v_{m,m-1})+f(v_{m+1,m+1})}=\lambda^-f(v_{m,m}).
\end{eqnarray}
Let us denote $\alpha=\frac{\lambda^+}{q}$, $\beta=\frac{\lambda^-}{q}$ and $a_{m,n}=\frac{f(v_{m,n})}{q^{m+n}}$. Using \eqref{8} and \eqref{9}, for any $m\geq 2,$ we have
\begin{eqnarray}
\label{12}&&a_{m+1,0}+q(q+1)a_{m,1}=\alpha a_{m,0}\\
\label{13}&&a_{m-2,0}+q(q+1)a_{m,1}=\beta a_{m-1,0}
\end{eqnarray}
It follows from \eqref{12} and \eqref{13} that for any $m\geq 2,$
\begin{equation}
a_{m+1,0}-\alpha a_{m,0}+\beta a_{m-1,0}+a_{m-2,0}=0.
\end{equation}
Since the point $(s_1,s_2,s_3)$ is the solution of the equation $X^3-\alpha X^2+\beta X-1$, $a_{m,0}$ is of the form
$$a_{m,0}=A_{1,0}s_1^m+A_{2,0}s_2^m+A_{3,0}s_3^m.$$
The numbers $A_{1,0}$, $A_{2,0}$ and $A_{3,0}$ are the solution of the following system of equations 
\begin{eqnarray}
\nonumber &&a_{0,0}=A_{1,0}+A_{2,0}+A_{3,0}=1\\
\nonumber &&a_{1,0}=A_{1,0}s_1+A_{2,0}s_2+A_{2,0}s_3=\frac{s_1+s_2+s_3}{q^2+q+1}\\
\nonumber &&a_{2,0}=A_{1,0}s_1^2+A_{2,0}s_2^2+A_{3,0}s_3^2=\frac{(s_1+s_2+s_3)^2-(q+1)(s_1s_2+s_2s_3+s_3s_1)}{q^2+q+1}.
\end{eqnarray}
Thus for any $i\in \{1,2,3\},$ we have
$$A_{i,0}=\frac{(s_i-qs_{j})(s_i-qs_{k})}{(s_i-s_{j})(s_i-s_{k})(q^2+q+1)},$$
where $j\neq k\in\{1,2,3\}\backslash \{i\}. $ 
The equation \eqref{13} shows that for any $m\geq 1,$ 
$$a_{m,1}=A_{1,1}s_0^m+A_{2,1}s_1^m+A_{3,1}s_2^m,$$
where $$A_{i,1}:=\frac{1}{q^2+q}\frac{1}{s_i}(\beta-\frac{1}{s_i})A_{i,0}=\frac{s_{j}+s_{k}}{q^2+q}A_{i,0}.$$
For $m>n\geq 1$, we have
\begin{eqnarray}
A_w^+f(v_{m,n})&=&{q^2f(v_{m-1,n-1})+qf(v_{m,n+1})+f(v_{m+1,n})}=\lambda^+f(v_{m,n})\\
A_w^-f(v_{m,n})&=&{q^2f(v_{m-1,n})+qf(v_{m,n-1})+f(v_{m+1,n+1})}=\lambda^-f(v_{m,n}).
\end{eqnarray}
This implies that for any $m>n\geq 1,$
\begin{eqnarray}
\label{17}&&\alpha a_{m,n}=\frac{1}{q}a_{m-1,n-1}+qa_{m,n+1}+a_{m+1,n}\\
\label{18}&&\beta  a_{m,n}=a_{m-1,n}+\frac{1}{q}a_{m,n-1}+qa_{m+1,n+1}.
\end{eqnarray}
Denote $a_{m,n}=A_{n,1}s_1^m+A_{n,2}s_2^m+A_{n,3}s_3^m.$

It follows from \eqref{17} that for any $n\in \mathbb{N}$, $i\in\{1,2,3\}$ and distinct $j, k\in \{1,2,3\}\backslash\{i\},$
 $$q^2A_{n+1,i}-q(s_{j}+s_{k})A_{n,i}+s_{j}s_{k}A_{n-1,i}=0.$$
 Since $\frac{s_{j}}{q}$ and $\frac{s_{k}}{q}$ are the solution of the equation of $q^2X^2-q(s_{j}+s_{k})X+s_{j}s_{k},$ we have
 $$A_{i,n}=B_{i,j}\frac{s_{j}^n}{q^n}+B_{i,k}\frac{s_{k}^n}{q^n}.$$
 Since we know $A_{i,0}$ and $A_{i,1},$ for any distinct $j,k\in\{1,2,3\}\backslash\{i\},$
$$B_{i,j}=\frac{(s_{i}-qs_{j})(s_i-qs_{k})(s_j-qs_{k})}{(s_{i}-s_j)(s_{i}-s_k)(s_{j}-s_{k})(q+1)(q^2+q+1)}.$$ 
Thus the eigenfunction $f$ is defined by
$$f(v_{m,n})=q^{m+n}a_{m,n}=\sum_{\substack{(i,j)\in \{1,2,3\}^2\\i\neq j}}B_{i,j}q^ms_i^ms_{j}^n.$$
This completes the proof.
\end{proof}

The following proposition gives the formula of simultaneous eigenfunctions in the \emph{singular cases}, that is, when the numbers $s_j$ are not distinct. It is possible to solve the recurrence relations by appropriate modifications of the above methods. Rather than this, we obtain the desired formula by taking limits of the nonsingular case.

\begin{prop}
Let $f$ be an eigenfunction of $A_w^+$ and $A_w^-$ with eigenvalues $\lambda^+$ and $\lambda^-$. If $s_1\ne s_2=s_3$, then 
\begin{align*}f(v_{m,n})=\frac{q^m}{(s_1-s_2)^2(q+1)(q^2+q+1)}&\biggl[\{(s_1-qs_2)^2\{(1-q)n+(q+1)\}s_1^ms_2^n\\
+&\{(s_2-qs_1)^2\{(1-q)(m+n)+(q+1)\}s_2^{m+n}\\
+&m(q-1)(s_1-qs_2)(s_2-qs_1)s_1^ns_2^m\\
+&(q+1)\{q(s_1^2+s_2^2)-2(q^2+q+1)s_1s_2\}s_1^ns_2^{m}\biggr].
\end{align*}
If $s_1=s_2=s_3$, then 
\begin{align*}
f(v_{m,n})=&\frac{s_1^{m+n}q^m}{2(q^2+q+1)}\biggl\{2(q+1)(q^2+q+1)+6m(1-q^2)\\&+(1-q)^2(q+1)(m^2-(2n-3)m-2n^2)+(1-q)^3(m^2n-mn^2)\biggr\}.
\end{align*}
\end{prop}

\begin{proof}
Assume first that $s_1\ne s_2=s_3$. Let us consider the following three limits:
\begin{equation}
\nonumber\lim_{s_3\rightarrow s_2} B_{1,2}s_2^n+B_{1,3}s_3^n=\frac{(s_1-qs_2)^2\{(1-q)n+(q+1)\}}{(s_1-s_2)^2(q+1)(q^2+q+1)}s_2^n,
\end{equation} 
\begin{equation}
\nonumber\lim_{s_3\rightarrow s_2} B_{2,3}s_2^ms_3^n+B_{3,2}s_2^ns_3^m=\frac{(s_2-qs_1)^2\{(1-q)(m+n)+(q+1)\}}{(s_1-s_2)^2(q+1)(q^2+q+1)}s_2^{m+n},
\end{equation} 
and
\begin{equation}
\begin{split}
\nonumber\lim_{s_3\rightarrow s_2} B_{2,1}s_2^m+B_{3,1}s_3^m=&m\frac{(q-1)(s_1-qs_2)(s_2-qs_1)}{(s_1-s_2)^2(q+1)(q^2+q+1)}s_2^{m},\\
&+\frac{(q+1)\{q(s_1^2+s_2^2)-2(q^2+q+1)s_1s_2\}}{(s_1-s_2)^2(q+1)(q^2+q+1)}s_2^{m}.
\end{split}
\end{equation} 
The above equations and L'H\^opital's law show that 
\begin{equation}\label{eq:4.19}
\begin{split}
f(v_{m,n})=\frac{q^m}{(s_1-s_2)^2(q+1)(q^2+q+1)}&\biggl[\{(s_1-qs_2)^2\{(1-q)n+(q+1)\}s_1^ms_2^n\\
+&\{(s_2-qs_1)^2\{(1-q)(m+n)+(q+1)\}s_2^{m+n}\\
+&m(q-1)(s_1-qs_2)(s_2-qs_1)s_1^ns_2^m\\
+&(q+1)\{q(s_1^2+s_2^2)-2(q^2+q+1)s_1s_2\}s_1^ns_2^{m}\biggr].
\end{split}
\end{equation}
Now assume that $s_1=s_2=s_3$.
 The right hand side of  \eqref{eq:4.19} and its first derivative are zero when $s_1=s_2$. Using L'H\^opital's law twice, we have
\begin{equation}\label{eq:4.20}
\begin{split}
f(v_{m,n})=&\frac{s_1^{m+n}q^m}{2(q^2+q+1)}\biggl\{2(q+1)(q^2+q+1)+6m(1-q^2)\\&+(1-q)^2(q+1)(m^2-(2n-3)m-2n^2)+(1-q)^3(m^2n-mn^2)\biggr\}.
\end{split}
\end{equation}
which completes the proof of the proposition.
\end{proof}


\section{Automorphic spectra of the weighted adjacency operators}\label{sec:6}
In this section, we investigate the automorphic spectra of the weighted adjacency operators $A^\pm_w$ on $L^2_w(\Gamma\backslash\mathcal{B}(G))$. We will prove that there are no discrete spectrum except trivial eigenvalues. Furthermore, we will verify that there are continuous spectra of $A_w^\pm$ which are outside of the spectrum of $A^\pm$ on $L^2(\mathcal{B}(G)).$

Fix $\mathbf{s}=(s_1,s_2,s_3)\in S$. Let us denote by $f_{\mathbf{s}}$ the simultaneous eigenfunction of $A_w^{\pm}$ with eigenvalues defined in \eqref{eq:4.1}.
\begin{prop}\label{prop:5.1}
The trivial eigenfunctions are the only eigenfunctions $f_\mathbf{s}$ of $A_w^\pm$ which belongs to $L^2_w(\Gamma\backslash\mathcal{B}(G))$. 
\end{prop}
\begin{proof} 
By definition, we have
\begin{equation}\label{eq:5.1}
\begin{split}
\|f_\mathbf{s}\|^2_2&=\sum_{m=0}^\infty\sum_{n=0}^m |f(v_{m,n})|^2w(v_{m,n}).
\end{split}
\end{equation}
Suppose that $s_1=s_2=s_3$. By assumption, we must have $|s_1|=1$. Since the sequence $|f_\mathbf{s}(v_{m,0})|^2w(v_{m,0})$ is a polynomial in $m$ by \eqref{eq:4.20} and \eqref{eq:3.3}, it follows that $\|f_\mathbf{s}\|^2_2=\infty$. 

Let us consider the case when $s_1\neq s_2=s_3$. By Lemma \ref{b}, $|s_1|=|s_2|=1.$ Since
$$f(v_{m,n})=As_1^m+(B+Cm)s_2^m$$
for some nonzero constants $A,B$ and $C$, $|f_\mathbf{s}(v_{m,0})|^2w(v_{m,0})$ diverges and $\|f_\mathbf{s}\|^2_2=\infty$.

In the case $s_1\neq s_2 \neq s_3$ and $|s_1|=|s_2|=|s_3|=1$, 
$$|f_\mathbf{s}(v_{m,0})|^2w(v_{m,0})=|(A_{1,0}s_1^m+A_{2,0}s_2^m+A_{3,0}s_3^m)|^2.$$
Since every $A_{i,0}$ is nonzero, $|f_\mathbf{s}(v_{m,0})|^2w(v_{m,0})$ does not converge to zero as $m$ goes to infinity and $\|f_\mathbf{s}\|=\infty.$
Suppose that $|s_1|>|s_2|>|s_3|$. Since $|s_1|>1$, $|s_2|=1$ and $|s_3|=|s_1|^{-1}$ by Lemma \ref{b}, if $B_{1,2}$ or $B_{2,1}$ is nonzero, $|f_\mathbf{s}(v_{m,m})|^2w(v_{m,m})$ diverges as $m$ goes to infinity and $\|f_\mathbf{s}\|_2=\infty$. The 3-tuple $(s_1,s_2,s_3)$ satisfies $s_1=qs_2$ and $s_2=qs_3$  when $B_{1,2}=B_{2,1}=0$. The 3-tuple $(s_1,s_2,s_3)$ is of the form $e^{\frac{2k\pi i}{3}}(q,1,1/q)$ for any $k\in \mathbb{Z}$. Hence the eigenfunctions are described by
\begin{equation}
\nonumber f_\mathbf{s}(v_{m,n})=\omega^{m+n}\quad (\omega=e^{\frac{2k\pi i}{3}}).
\end{equation}
The remaining part is to show that the above function is $L^2$. In fact,
\begin{equation}
\begin{split}
\|f_{\mathbf{s}}\|^2_2=&\sum_{m=0}^\infty\sum_{n=0}^m |f_\mathbf{s}(v_{m,n})|^2w(v_{m,n})\\
=&\frac{B_{3,2}^2}{q^2+q+1}+\sum_{m=1} |f_\mathbf{s}(v_{m,0})|^2w(v_{m,0})+\sum_{m=1}^\infty\sum_{n=1}^{m-1} |f_\mathbf{s}(v_{m,n})|^2w(v_{m,n})\\
&+\sum_{m=1}|f_\mathbf{s}(v_{m,m})|^2w(v_{m,m})\\
=&\frac{1}{q^2+q+1}+2\sum_{m=1} q^{-2m} +\sum_{m=1}^\infty\sum_{n=1}^{m-1} (q+1)q^{-2m}\\
=&\frac{1}{q^2+q+1}+2\sum_{m=1} q^{-2m} +\sum_{m=1}^\infty(m-1)(q+1)q^{-2m}<\infty.
\end{split}
\end{equation}
This yields the proof of the proposition.
\end{proof}
From the proof of the above statement, we obtain the following two corollaries.
\begin{coro} Every nontrivial eigenfunction $f_{\mathbf{s}}$ of $A_w^\pm$ indexed by $\mathbf{s}\in S$ is not in $L^2(\Gamma\backslash\mathcal{B}(G)).$
\end{coro}
\begin{coro}
The operators $A_w^\pm$ are bounded operator on $L^2_w(\Gamma\backslash\mathcal{B}(G))$ with operator norm $q^2+q+1.$
\end{coro}
\begin{proof}
If $\mathbf{s}=(q,1,\frac{1}{q})$, the eigenfunction $f_{\mathbf{s}}\equiv 1$ is positive and $A_w^\pm f_{\mathbf{s}}=(q^2+q+1) f_{\mathbf{s}}$. By Schur's test (see \cite{Pe}, page 102), the operator norm is $q^2+q+1.$ 
\end{proof}
Given $(s_1,s_2,s_3)\in S$ and $\epsilon>0$, let us define the function $f_{\mathbf{s}}^\epsilon$ by
$$f_\mathbf{s}^\epsilon(v_{m,n})=(1-\epsilon)^{m}f_\mathbf{s}(v_{m,n}).$$
\begin{lem}\label{lem:6.4} Let $\mathbf{s}$ be a point in $S$ such that $s_1$, $s_2$ and $s_3$ are distinct and $|s_i|=1$ for any $ i\in\{1,2,3\}$. For any $\epsilon\in (0,1/2)$, $f_\mathbf{s}^\epsilon$ is $L^2$-function and 
\begin{equation}\label{eq:5.3}
\lim_{\epsilon \rightarrow 0} \frac{\|A_w^\pm f_\mathbf{s}^\epsilon-\lambda^\pm f_\mathbf{s}^\epsilon\|_2}{\|f_\mathbf{s}^\epsilon\|_2}=0,
\end{equation}
where $\lambda^\pm=q(s_1^{\pm1}+s_2^{\pm1}+s_3^{\pm1}).$ The norm $\|f_\mathbf{s}^\epsilon\|_2$ goes infinity as $\epsilon$ goes to zero.
\end{lem}
\begin{proof}
Since 
\begin{equation}\label{eq:5.4}
|f_{\mathbf{s}}(v_{m,n})|\leq q^m\sum_{\substack{i,j\\i\neq j}}|B_{i,j}|
\end{equation}
hold for any $m\ge n\ge 0$, it follows that
$f_{\mathbf{s}}^\epsilon$ is in $L^2_w(\Gamma\backslash\mathcal{B}(G))$. By Proposition \ref{eq:5.1}, $\|f_\mathbf{s}^\epsilon\|_2$ goes to infinity as $\epsilon$ goes to zero.

 It remains to show that \eqref{eq:5.3}. Since $\lambda^+f_\mathbf{s}^\epsilon(v_{m,n})=(1-\epsilon)^mA_w^+f_{\mathbf{s}}(v_{m,n}),$ we have
\begin{align*}\label{eq:5.5}
  A_w^+f_\mathbf{s}^\epsilon(v_{0,0})-\lambda^+f_\mathbf{s}^\epsilon(v_{0,0})=& A_w^+ f_\mathbf{s}^\epsilon(v_{0,0})-A_w^+f_\mathbf{s}(v_{0,0})\\
 =&(q^2+q+1)\{(1-\epsilon)f(v_{1,0})-f(v_{1,0})\}\\
 =&-\epsilon(q^2+q+1) f(v_{1,0})\\
 A_w^+f_\mathbf{s}^\epsilon(v_{m,0})-\lambda^+f_\mathbf{s}^\epsilon(v_{m,0})=&A_w^+f_\mathbf{s}^\epsilon(v_{m,0})-(1-\epsilon)^m A_w^+f_\mathbf{s}(v_{m,0})\\
 =&(1-\epsilon)^m(q^2+q)f(v_{m,1})+(1-\epsilon)^{m+1}f(v_{m+1,0})\\
 &-(1-\epsilon)^m(q^2+q)f(v_{m,1})-(1-\epsilon)^{m}f(v_{m+1,0})\\
 =&-\epsilon(1-\epsilon)^mf(v_{m+1,0})\\ 
 A_w^+f_\mathbf{s}^\epsilon(v_{m,m})-\lambda^+f_\mathbf{s}^\epsilon(v_{m,m})=&A_w^+f_\mathbf{s}^\epsilon(v_{m,m})-(1-\epsilon)^m A_w^+f_\mathbf{s}(v_{m,m})\\ \displaybreak[0]
 =&(1-\epsilon)^{m-1}q^2(v_{m-1,m-1})+(1-\epsilon)^{m+1}(q+1)f(v_{m+1,m})\\
 &-(1-\epsilon)^m q^2f(v_{m-1,m-1})-(1-\epsilon)^{m}(q+1)f(v_{m+1,m})\\
 =&\epsilon(1-\epsilon)^{m-1}q^2f(v_{m-1,m-1})-\epsilon(1-\epsilon)^{m}(q+1)f(v_{m+1,m})\\
A_w^+ f_\mathbf{s}^\epsilon(v_{m,n})-\lambda^+f_\mathbf{s}^\epsilon(v_{m,n})=&A_w^+f_\mathbf{s}^\epsilon(v_{m,n})-(1-\epsilon)^m A_w^+f_\mathbf{s}(v_{m,n})\\
 =& (1-\epsilon)^{m-1}q^2f(v_{m-1,n})+(1-\epsilon)^mqf(v_{m,n-1})\\
 &+(1-\epsilon)^{m+1}f(v_{m+1,n+1})-(1-\epsilon)^mq^2f(v_{m-1,n})\\
 &-(1-\epsilon)^mqf(v_{m,n-1})-(1-\epsilon)^mf(v_{m+1,n+1})\\
 =&\epsilon(1-\epsilon)^{m-1}q^2f(v_{m-1,n})-\epsilon(1-\epsilon)^mf(v_{m+1,n+1}).
\end{align*}
 Denote $B:=\sum_{\substack{i,j\\i\neq j}}|B_{i,j}|.$
 By \eqref{eq:5.4} and the above formulas together with the triangle inequality, we have
 \begin{equation}\label{eq:5.6}
 \begin{split}
 &|A_w^+f_\mathbf{s}^\epsilon(v_{0,0})-\lambda^+f_\mathbf{s}^\epsilon(v_{0,0})|\leq \epsilon q(q^2+q+1)B\\
 &|A_w^+f_\mathbf{s}^\epsilon(v_{m,0})-\lambda^+f_\mathbf{s}^\epsilon(v_{m,0})|\leq\epsilon(1-\epsilon)^{m}q^{m+1} B\\
 &|A_w^+f_\mathbf{s}^\epsilon(v_{m,m})-\lambda^+f_\mathbf{s}^\epsilon(v_{m,m})|\leq2\epsilon(1-\epsilon)^{m-1}(q^2+q+1)q^{m+1}B\\
 &|A_w^+ f_\mathbf{s}^\epsilon(v_{m,n})-\lambda^+f_\mathbf{s}^\epsilon(v_{m,n})|\leq2\epsilon(1-\epsilon)^{m-1}(q^2+1)q^{m+1}B.
 \end{split}
 \end{equation}
  Using \eqref{eq:5.6}, we have
  \begin{align*}
  &\,\,\|A_w^\pm f_\mathbf{s}^\epsilon-\lambda^\pm f_\mathbf{s}^\epsilon\|_2^2=\sum_{m=0}^\infty\sum_{n=0}^m|A_w^+f_\mathbf{s}^\epsilon(v_{m,n})-\lambda^+f_\mathbf{s}^\epsilon(v_{m,n})|^2w(v_{m,n})\\
  =&\,\,|A_w^+f_\mathbf{s}^\epsilon(v_{0,0})-\lambda^+f_\mathbf{s}^\epsilon(v_{0,0})|^2w(v_{0,0})+\sum_{m=1}^\infty|A_w^+f_\mathbf{s}^\epsilon(v_{m,0})-\lambda^+f_\mathbf{s}^\epsilon(v_{m,0})|^2w(v_{m,0})\\ \displaybreak[0]
 &+\sum_{m=1}^\infty|A_w^+f_\mathbf{s}^\epsilon(v_{m,m})-\lambda^+f_\mathbf{s}^\epsilon(v_{m,m})|^2w(v_{m,m}) \\ \displaybreak[0]
 &+\sum_{m=1}^\infty\sum_{n=1}^{m-1}|A_w^+ f_\mathbf{s}^\epsilon(v_{m,n})-\lambda^+f_\mathbf{s}^\epsilon(v_{m,n})|^2w(v_{m,n})\\ \displaybreak[0]
 \leq&\,\,\epsilon^2(q^2+q+1)q^2B^2+\epsilon^2q^2B^2\sum_{m=1}^\infty(1-\epsilon)^{2m}+4\epsilon^2q^2(q^2+q+1)^2B^2\sum_{m=1}^\infty(1-\epsilon)^{2m-2}\\ \displaybreak[0]
 &+4\epsilon^2q^2(q+1)(q^2+1)^2B^2\sum_{m=1}^\infty (m-1)(1-\epsilon)^{2m-2}\\
 =&\,\,\epsilon^2q^2(q^2+q+1)B^2+\frac{\epsilon^2q(1-\epsilon)^2B^2}{2\epsilon-\epsilon^2}+\frac{4\epsilon ^2q^2(q^2+q+1)^2B^2}{2\epsilon-\epsilon^2}\\ \displaybreak[0]
 &+\frac{4\epsilon^2q^2(q+1)(q^2+1)^2B^2(1-\epsilon)^2}{(2\epsilon-\epsilon^2)^2}\\
 \leq&\,\,\epsilon^2q(q^2+q+1)B^2+{\epsilon qB^2}+{4\epsilon q(q^2+q+1)^2B^2}+{4q(q+1)(q^2+1)^2B^2}.
 \end{align*}
   Since $\|f_\mathbf{s}^\epsilon\|_2$ goes to infinity as $\epsilon$ goes to zero, from the above inequality we have \eqref{eq:5.4}.
\end{proof}

\begin{coro}\label{coro:6.5} Let $\mathbf{s}$ be a point in $S$ such that $s_1>s_2>s_3$. For any $\epsilon\in (0,1/2)$, $f_\mathbf{s}^\epsilon$ is $L^2$-function and satisfies \eqref{eq:5.3} only if $\mathbf{s}=(\sqrt{q}e^{i\theta},e^{-2i\theta},\frac{e^{i\theta}}{\sqrt{q}})$ for any $\theta\in \mathbb{R}$ or $e^{\frac{2\pi k}{3}i}(q,1,1/q)$. 
\end{coro}
\begin{proof}
As in the proof of Proposition \ref{prop:5.1}, for any $\mathbf{s}\in S$ with $|s_1|>|s_2|>|s_3|$, the function $f_\mathbf{s}^\epsilon$ is not in $L^2_w(\Gamma\backslash\mathcal{B}(G))$ for some $\epsilon \in (0,1/2)$ unless $B_{1,2}=B_{1,3}=B_{2,1}=0$. This holds when $s_1=qs_2$ and $s_2=qs_3$, or $s_1=qs_3$. By Lemma \ref{b}, $\mathbf{s}=(q\omega,\omega,\frac{\omega}{q})$ for $\omega^3=1$ or $\mathbf{s}=(\sqrt{q}e^{i\theta},e^{-2i\theta},\frac{e^{i\theta}}{\sqrt{q}})$ for $\theta\in \mathbb{R}$. Since Proposition \ref{prop:5.1} deals with the first case, it is enough to consider the second case. In this case, we also have \eqref{eq:5.4}. Following the proof of Lemma \ref{lem:6.4}, we obtain Corollary \ref{coro:6.5}.
\end{proof}

This gives the proof of Theorem~\ref{thm:1.1}.

\begin{proof}[Proof of Theorem \ref{thm:1.1}]By Lemma \ref{prop:5.1}, the discrete spectrum of $A_w^+$ on $L^2_w(\Gamma\backslash\mathcal{B}(G))$ is $\Sigma_0$. Lemma \ref{lem:6.4} implies that $\Sigma_2$ is in the automorphic spectra of $A_w^+$. Corollary \ref{coro:6.5} shows that $\Sigma_1$ is also contained in the automorphic spectra of $A_w^+$ on $L^2_w(\Gamma\backslash\mathcal{B}(G))$.
\end{proof}

\end{document}